\documentclass[11pt,leqno]{amsart}
\usepackage{amsmath,amssymb,amsthm}
\usepackage{hyperref}

\usepackage{tikz-cd}

\usepackage{hyperref}

\usepackage{bbm}

\DeclareMathOperator{\diam}{diam\,}
\DeclareMathOperator{\co}{co}

\renewcommand{\geq}{\geqslant}
\renewcommand{\leq}{\leqslant}

\newcommand{\norm}[1]{\left\Vert#1\right\Vert}

\newcommand{\spann}{\operatorname{span}}

\newcommand{\1}{\textbf{1}}

\newtheorem{theorem}{Theorem}[section]
\newtheorem{lemma}[theorem]{Lemma}

\newtheorem{proposition}[theorem]{Proposition}

\theoremstyle{definition}
\newtheorem{definition}[theorem]{Definition}
\newtheorem{example}[theorem]{Example}

\theoremstyle{remark}
\newtheorem{remark}[theorem]{Remark}
\newtheorem{question}{Question}
\numberwithin{equation}{section}

\def\fnote#1{\footnote}

\def\natu{{\mathbb N}}

\def\ignora#1{}
\def\n3#1{\left\vert  \! \left\vert \! \left\vert \, #1 \, \right\vert \!
  \right\vert \! \right\vert }

\renewcommand{\leq}{\le}

\usepackage{accents}
\usepackage[most]{tcolorbox}
\usepackage{comment}

\let\emptyset\varnothing

\newcommand{\N}{\mathbb{N}}

\newcommand{\U}{\mathcal{U}}

\newcommand{\eps}{\varepsilon}
\newcommand{\conv}{\text{co}}

\begin{document}

\author{ Abraham Rueda Zoca }\address{Universidad de Granada, Facultad de Ciencias. Departamento de An\'{a}lisis Matem\'{a}tico, 18071-Granada
(Spain)} \email{ abrahamrueda@ugr.es}
\urladdr{\url{https://arzenglish.wordpress.com}}

\subjclass[2020]{46B04; 46B08; 46B20; 46M07}

\keywords{Diameter two properties; ultrapower spaces; coultably incomplete ultrafilters; renorming}

\title{Differences between the uniform diameter two properties}

\begin{abstract}
We study uniform versions of the diameter two properties, that is, the uniform slice-D2P, the uniform D2P and the uniform SD2P, as the property that every ultrapower over any free ultrafilter over $\mathbb N$ has the respective diameter two property. The aim of this note is to prove that all the uniform diameter two properties are different eachother.
\end{abstract}

\maketitle

\section{Introduction}

We recall that a Banach space $X$ has the \textit{slice diameter two property (slice-D2P)} (respectively the \textit{diameter two property (D2P)}, the \textit{strong diameter two property (SD2P)}) if every slice (respectively non-empty relatively weakly open subset, convex combination of slices) of the unit ball has diameter two. It is evident that the D2P implies the slice-D2P since every slice of $B_X$ is a non-empty relatively weakly open subset. Furthermore, since every non-empty relatively weakly open subset of $B_X$ contains a convex combination of slices of $B_X$ thanks to a result of Bourgain \cite[Lemma II.1]{ggms}, the SD2P implies the D2P. It is known that none of the above implication reverse \cite{blr15eje2}. We refer the reader to \cite{ahlt26,ahntt16,almt21,blr15eje1,blr15eje2,nw01,rueda26} and references therein for background about the diameter two properties.

The literature on the diameter two properties is vast and it has been analysed in multiple classes of Banach spaces such as spaces of functions \cite{bl06}, tensor product spaces \cite{blr4,hlp17,rueda26} or in spaces of Lipschitz functions and Lispchitz-free spaces \cite{gpr18,hop,lr2020}. The study of the diameter two properties in ultrapower spaces, however, remained unnoticed until \cite{rz25}. Before going on, let us introduce the following notation: given a Banach space $X$, we say that $X$ has the \textit{uniform slice diameter two property (uniform slice-D2P)} (respectively the \textit{uniform diameter two property (uniform D2P)}, the \textit{uniform strong diameter two property (uniform SD2P)} if $X_\mathcal U$ has the slice-D2P (respectively the D2P, the SD2P) for any free ultrafilter $\mathcal U$ over $\mathbb N$.

Coming back to the above mentioned paper \cite{rz25}, a complete characterisation of the uniform slice-D2P was obtained making use of the following result of Y. Ivakhno \cite[Lemma 1]{iva06}: a Banach space $X$ has the slice-D2P if, and only if, $B_X=\overline{\conv}\left\{\frac{x+y}{2}: x,y\in B_X, \Vert x-y\Vert\geq 2-\varepsilon\right\}$ holds for every $\varepsilon>0$. Thanks to the above mentioned result by Ivakhno, which is a characterisation of the slice-D2P that avoids the access to the topological dual, the uniform slice-D2P in a Banach space $X$ could be described in \cite[Corollary 3.4]{rz25} in terms of a geometric property of the underlying space $X$ (see Subsection~\ref{subsect:diametertwo} for details). This permitted to provide a number of examples of spaces with the uniform slice-D2P. Moreover, the above result was also used in order to show that, for every $\varepsilon>0$, there exists a Banach space $X$ with the slice-D2P and such that $X_\mathcal U$ contains slices of diameter smaller than $\varepsilon$ for any free ultrafilter $\mathcal U$
over $\mathbb N$ \cite[Theorem 4.6]{rz25} (let us point out that the above example follows literally the construction of a Banach space $X$ with the Daugavet property which fails the uniform Daugavet property by V. Kadets and D. Werner from \cite{kw04}). In the same spirit as in \cite{rz25}, in the recent paper \cite{mr26uni} the uniform SD2P is characterised by making use of a characterisation of the SD2P which does not make use of the access to the dual space \cite[Theorem 3.6]{mr26uni}: a Banach space $X$ has the SD2P if, and only if, for every $n \in \N$, $x_1, \cdots, x_n \in B_X$ and $\eps >0$, there exist an $m \in \N$, some $y_{ij} \in B_X$ for $i=1, \cdots, n$, $j=1, \cdots, m$ and $\alpha_1, \cdots, \alpha_m >0$  with $\sum_{j=1}^m \alpha_j =1$ such that $\norm{x_i - \sum_{j=1}^m \alpha_j y_{ij}} < \eps$ holds for every $i=1, \cdots, n$ and $\norm{\sum_{i=1}^n \dfrac{1}{n} y_{ij}}> 1-\eps$ holds  for every $j=1, \cdots, m$. The above reformulation allowed to get a characterisation of the uniform SD2P in \cite[Theorem 3.7]{mr26uni} which permited to obtain a number of examples of spaces with the uniform SD2P. Concerning the uniform D2P, up to the best of the author knowledge, no systematic study has been done of this property. Observe that any of the uniform diameter two property implies its non uniform version by the arguments of \cite[Proposition 3.11]{mr26uni}.

In view of the fact that the 3 diameter two properties are different, a natural question is whether the uniform versions are indeed different. Taking into account the well known description of the diameter two properties by taking $\ell_p$ sums \cite{abl15}, and thanks to the isometric identification $X_\mathcal U \oplus_p Y_\mathcal U=(X\oplus_p Y)_\mathcal U$, it was pointed out in \cite[Remark 3.10]{mr26uni} that the uniform D2P and the uniform SD2P are different properties. The question whether the uniform slice-D2P and the uniform D2P are different is open (also asked in the above mentioned \cite[Remark 3.10]{mr26uni}). Observe that the main difficulty for facing this problem is that, contrary to what happen for the slice and strong version, no description of the D2P in a ultrapower space $X_\mathcal U$ is known in terms of a geometric property of the base space $X$.

The first aim of this note is to provide examples of Banach spaces with the uniform slice-D2P but failing even the D2P. Indeed, the first main theorem of the paper is the following.

\begin{theorem}\label{theo:maintheodiffslid2punif}
Let $X$ be a Banach space containing an isomorphic copy of $c_0$. Then there exists an equivalent renorming on $X$ with the uniform slice-D2P but such that the unit ball contains non-empty relatively weakly open subsets of arbitrarily small diameter. \end{theorem}

In order to prove the above theorem we will perform the renorming from \cite[Theorem 2.4]{blr15eje1}, where it is proved that every Banach space containing $c_0$ admits an equivalent renorming with the slice-D2P but whose unit ball contains non-empty relatively weakly open subsets of arbitrarily small diameter. We will prove that this equivalent renorming satisfies the uniform slice-D2P by a direct application of \cite[Corollary 3.4]{rz25}. In order to do so, we will introduce the set $STS$ in Section~\ref{section:STS}, because such set is strongly employed in the renorming technique from \cite[Theorem 2.4]{blr15eje1}, and we will prove a result of uniform approximation in Proposition~\ref{prop:aprouniSTS} which is critically used in order to prove Theorem~\ref{theo:maintheodiffslid2punif} in Section~\ref{section:prooftheoremunislice}.

Even though it is known that the uniform D2P and the uniform SD2P are different properties, in view of the statement of Theorem~\ref{theo:maintheodiffslid2punif} and the statement of \cite[Theorem 2.5]{blr15eje2}, it is natural to wonder whether there are Banach spaces with the uniform D2P and such that the unit ball contains convex combinations of slices of arbitrarily small diameter. 

In view of the above mentioned identification $X_\mathcal U \oplus_p Y_\mathcal U=(X\oplus_p Y)_\mathcal U$, a natural way to get one such counterexample could be to consider an infinite $\ell_p$-sum of Banach spaces with the uniform SD2P, because in one such space there are convex combinations of slices of the unit ball of arbitrarily small diameter (see Proposition~\ref{prop:ccslicearbipe}). However, we are unable to follow this circle of ideas because of two reasons: on the one hand; the identification $X_\mathcal U \oplus_p Y_\mathcal U=(X\oplus_p Y)_\mathcal U$ fails to be true if we consider infinitely many Banach spaces (see Example~\ref{exam:ultrapower}) and; on the other hand, we do not know whether the uniform D2P is inherited by infinite $\ell_p$-sums because, as we have already indicated, we do not know any precise description of the uniform D2P in a Banach space. 

Because of all the above reasons, we will consider a different approach: we will consider the renorming technique from \cite[Theorem 2.5]{blr15eje2} and we will get the second main theorem of the paper.

\begin{theorem}\label{theo:maintheodiffd2punif}
Let $X$ be a Banach space containing an isomorphic copy of $c_0$. Then there exists an equivalent renorming on $X$ with the uniform D2P but such that the unit ball contains convex combinations of slices of arbitrarily small diameter. 
\end{theorem}

In order to prove the above theorem we will perform the renorming from \cite[Theorem 2.5]{blr15eje2}, where it is proved that every Banach space containing $c_0$ admits an equivalent renorming with the slice-D2P but whose unit ball contains non-empty relatively weakly open subsets of arbitrarily small diameter. If we consider one such $X$ equipped with the above mentioned norm, our strategy will be to prove that $X_\mathcal U$ has the D2P for any free ultrafilter $\mathcal U$ over $\mathbb N$. The main idea here is to take advantage of the strong $c_0$ behaviour of this renorming in order to get appropriate weakly null sequences in a ultrapower space. In order to do so, we will present in Section~\ref{section:prelisegureno} some preliminary content for the proof of Theorem~\ref{theo:maintheodiffd2punif}: first, we will recall the definition of a subset $K_0\subseteq c_0$ which is behind the renorming technique from \cite[Theorem 2.5]{blr15eje2}; second, we will exhibit in Proposition~\ref{prop:condisufiD2P} a sufficient condition for a Banach space $X$ to enjoy the D2P in terms of weakly convergent sequences. With this preliminary material, we will prove Theorem~\ref{theo:maintheodiffd2punif} in Section~\ref{section:prooftheoremunid2p}. We will end the paper with some remarks in Section~\ref{section:remarks}.

\section{Notation and preliminary results}

We will consider real Banach spaces. Given a Banach space $X$, $B_X$ (respectively $S_X$) stands for the closed unit ball (respectively the unit sphere) of $X$. We will denote by $X^*$ the topological dual of $X$. Given a subset $C$ of $X$, we will denote by $\conv(C)$ the convex hull of $C$. We also denote by $\conv_m(C)$ the set of all convex combinations of at most $m$ elements of $C$, that is
$$\conv_m(C):=\left\{\sum_{i=1}^m \lambda_i x_i: \lambda_1,\ldots, \lambda_m\in [0,1], \sum_{i=1}^m \lambda_i=1, x_1,\ldots, x_m\in C \right\}.$$

If $C$ is a bounded subset of a Banach space $X$, by a \textit{slice} of $C$ we will mean a set of the following form
$$S(C,f,\alpha):=\{x\in C:  f(x)>\sup f(C)-\alpha\}$$
where $f\in X^*$ and $\alpha>0$. Notice that a slice is nothing but the non-empty intersection of a half-space with the bounded (and not necessarily convex) set $C$.

If $C$ is assumed to be convex we will mean by a \textit{convex combination of slices of $C$} a set of the following form
$$\sum_{i=1}^n \lambda_i S_i,$$
where $\lambda_1,\ldots, \lambda_n\in ]0,1]$ are such that $\sum_{i=1}^n \lambda_i=1$ and $S_i$ is a slice of $C$ for every $i\in\{1,\ldots, n\}$.

\subsection{Ultraproducts of Banach spaces}

Given a family of Banach spaces $\{X_\lambda: \lambda \in \Lambda\}$ we denote 
$$\ell_\infty(\Lambda,X_\lambda):=\left\{f \in \prod \limits_{\lambda\in \Lambda} X_\lambda: \sup_{\lambda \in \Lambda}\Vert f(\lambda)\Vert<\infty\right\},$$
and for any free ultrafilter $\mathcal U$ over $\Lambda$ we can consider the following closed subspace of $\ell_\infty(\Lambda, X_\lambda)$ 
$$c_{0,\mathcal U}(\Lambda,X_\lambda):= \left\{f\in \ell_\infty(\Lambda, X_\lambda): \lim_\mathcal U \Vert f(\lambda)\Vert=0 \right\}.$$
The \textit{ultraproduct of $\{X_\lambda: \lambda \in \Lambda\}$ with respect to $\mathcal U$} is the quotient Banach space
$$(X_\lambda)_\mathcal U:=\ell_\infty(\Lambda,X_\lambda)/c_{0,\mathcal U}(\Lambda,X_\lambda).$$
If every $X_\lambda$ is equal to some fixed Banach space $X$ we will call the ultraproduct of that family the \textit{ultrapower of $X$}, which will be denoted by $X_\U$.

We denote the image of $x \in \ell_\infty(\Lambda,X_\lambda)$ by some $\lambda \in \Lambda$ as $x_\lambda$ and by $[x_\lambda]_\mathcal U$, or simply $[x_\lambda]$ if there is no confusion, the coset $x + c_{0,\mathcal U}(\Lambda, X_\lambda) \in (X_\lambda)_\U$. Furthermore, from the definition of the quotient norm, it is not difficult to prove that
$$\Vert [x_\lambda]\Vert_\U=\lim_\mathcal U \Vert x_\lambda\Vert$$
for every $[x_\lambda] \in (X_\lambda)_\mathcal U$.  This implies that the canonical inclusion $j:X\rightarrow X_\mathcal U$ given by
$$j(x):=[x]_\mathcal U$$
is an into linear isometry. It is well known (c.f. e.g. Propositions 6.1 and 6.2 in \cite{hein80}) that $X_\mathcal U$ is finitely representable in $j(X)$. 

Given any subset $A\subseteq X$, we denote by
$$A_\mathcal U:=\{[x_i]\in X_\mathcal U: x_i\in A\ \forall i\in I\}.$$

Observe that, given $[x_i]\in X_\mathcal U$, it is always possible to choose a representative $[y_i]\in X_\mathcal U$ such that $[y_i]=[x_i]$ and $\Vert y_j\Vert\leq \lim_\mathcal U \Vert x_i\Vert=\Vert [x_i]\Vert$ holds for every $j\in I$ (see e.g. \cite[Remark 2.4]{mr26uni}). This proves in particular that $B_{(X)_\mathcal U}=(B_X)_\mathcal U$.

In the sequel we will recall a couple of properties of $A_\mathcal U$ that we will need. The first one is an easy behaviour with respect to finite unions.

\begin{proposition}\label{prop:ultraunion}
Let $I$ be an infinite set and $\mathcal U$ be a free ultrafilter over $I$. Let $X$ be a Banach space and let $A,B\subseteq X$. Then
$$(A\cup B)_\mathcal U=A_\mathcal U \cup B_\mathcal U.$$
\end{proposition}

\begin{proof}
It is evident that $\supseteq$ does hold since $A_\mathcal U$ and $B_\mathcal U$ are trivially subsets of $(A\cup B)_\mathcal U$. For the converse, select any $[x_i]\in (A\cup B)_\mathcal U$, so given $i\in I$ we have $x_i\in A\cup B$. Now define the sets
$$U:=\{i\in I: x_i\in A\}, V:=I\setminus U=\{i\in I: x_i\notin A\}.$$
Since $U,V$ form a partition of $I$ and $\mathcal U$ is a maximal filter then either $U$ or $V$ belongs to $\mathcal U$. 

If $U\in \mathcal U$ let us prove that $[x_i]\in A_\mathcal U$. Indeed, select any $a_0\in A$ and define
$$y_i:=\left\{\begin{array}{cc}
  x_i   &  i\in U \\
  a_0   & i\notin U
\end{array} \right.$$
It is immediate that $[y_i]\in A_\mathcal U$. Moreover, $[x_i]=[y_i]$. Indeed, given any $\varepsilon>0$, it follows that
$$U\subseteq A_\varepsilon:=\{i\in I: \Vert x_i-y_i\Vert<\varepsilon\}.$$
Since $U\in\mathcal U$ we infer $A_\varepsilon\in\mathcal U$. By the arbitrariness of $\varepsilon$ we get that $\Vert [x_i]-[y_i]\Vert=\lim_\mathcal U\Vert x_i-y_i\Vert=0$, and we are done.

In the case that $V\in\mathcal U$ we prove similarly that this time $[x_i]\in B_\mathcal U$. This proves the inclusion $\subseteq$ and finishes the proof.
\end{proof}

Another result relates the ultrapower of a set and the ultrapower of its closure. In order to establish it, recall that a ultrafilter $\mathcal U$ over an infinite set $I$ is said to be \textit{countably incomplete} if there exists a sequence $\{A_n\}\subseteq \mathcal U$ such that $\bigcap\limits_{n\in\mathbb N} A_n=\emptyset$. It is not difficult to prove that any countably incomplete ultrafilter $\mathcal U$ is free, and that any free ultrafilter over $\mathbb N$ is countably incomplete. It is worth mentioning that, given an infinite set $I$, the fact that every free ultrafilter over $I$ is countably incomplete is a frequent phenomenon. For instance, it is known \cite[Theorem 2.5]{hj} that if $\kappa$ is a cardinal with a free ultrafilter which is not countably incomplete, then $\kappa$ is a strongly inaccessible cardinal (see \cite{hj} for background).

Anyway, with the notion of countably incomplete ultrafilters let us present the last result of the subsection, whose proof can be found, for instance, in \cite[Proposition 1.2.5]{grelierdiss}.

\begin{proposition}\label{prop:ultracierre}
Let $X$ be a Banach space and let $A\subseteq X$ be a bounded set. Let $\mathcal U$ be a countably incomplete ultrafilter over an infinite set $I$. Then $A_\mathcal U$ is closed and
$$(\overline{A})_\mathcal U=A_\mathcal U=\overline{A_\mathcal U}.$$
\end{proposition}

\subsection{Diameter two properties and their uniform versions}\label{subsect:diametertwo}

In this section we will introduce a bit of notation and necessary preliminary results about the diameter two properties and their uniform versions.

As we have indicated in the Introduction of this manuscript none of the implications SD2P$\Rightarrow$D2P$\Rightarrow$slice-D2P reverse. Indeed, it was first proved in \cite[Theorem 2.4]{blr15eje1} that every Banach space containing $c_0$ admits an equivalent renorming with the slice-D2P but whose unit ball contains non-empty relatively weakly open subsets of arbitrarily small diameter. Similarly, it was proved in \cite[Theorem 2.5]{blr15eje2} that every Banach space containing $c_0$ admits an equivalent renorming with the D2P and whose unit ball contains convex combinations of slices of arbitrarily small diameter. In both results, the next lemma was essential to construct the equivalent renorming, whose statement is written below for future reference in the following sections.

\begin{lemma}\label{lemma:c0inicial}\cite[Lemma 2.3]{blr15eje1}
Let $X$ be a Banach space containing an isomorphic copy of $c_0$. Then there is an equivalent norm $||| \cdot |||$ on $X$ satisfying that $(X,|||\cdot|||)$ contains an isometric copy of $c$ and that for every
$x\in B_{(X,|||\cdot|||)}$ there are sequences $\{x_n\}$, $\{y_n\}\in B_{(X,|||\cdot|||)}$ which are weakly convergent to $x$ and such that $ ||| x_n -y_n ||| =2$ holds for every $n\in \natu$. In fact,
$x_n=x+(1-\alpha_n)e_n$ and $y_n=x-(1+\alpha_n)e_n$ for some
scalars sequence $\{\alpha_n\}$ satisfying that $\vert \alpha_n\vert\leq 1$ holds
for every $n$.
\end{lemma}

Let us now pass to providing a precise description of the characterisation of the uniform slice-D2P of \cite[Corollary 3.4]{rz25} because we will make use of it in Section~\ref{section:prooftheoremunislice}. In order to do so, let us borrow a bit of notation from \cite[Section 2]{rz25}: given a Banach space $X$ and $\alpha>0$, define
$$S^\alpha(X):=\left\{\frac{x+y}{2}: x,y\in B_X, \Vert x-y\Vert\geq  \alpha\right\}.$$
Given $n\in\mathbb N$ we denote
$$S_n^\alpha(X):=\conv_n(S^\alpha(X)).$$
Finally, given $n\in\mathbb N$ and $\alpha>0$, we define
$$C_n^\alpha(X):=\sup_{x\in S_X} d(x,S_n^\alpha(X))=\sup_{x\in S_X}\inf_{y\in S_n^\alpha(X)} \Vert x-y\Vert.$$
From the very definition of $C_n^\alpha(X)$ the following two properties follow:
\begin{enumerate}
    \item Given $0<\alpha<\beta$ then $C_n^\alpha\geq C_n^\beta$ and,
    \item given two natural numbers $n\geq m$ then $C_n^\alpha(X)\leq C_m^\alpha(X)$.
\end{enumerate}

Now the desired characterisation of the uniform slice-D2P is the following.

\begin{theorem}\label{theo:charunislice}\cite[Corollary 3.4]{rz25}
Let $X$ be a Banach space. The following are equivalent:
\begin{enumerate}
    \item $(X)_\mathcal U$ has the slice-D2P for every free ultrafilter $\mathcal U$ over $\mathbb N$.
    \item For every $0<\alpha<2$, $\lim_{n\rightarrow \infty} C_n^\alpha(X)=0$.
\end{enumerate}
\end{theorem}

Taking into account the thesis of both Theorems~\ref{theo:maintheodiffslid2punif} and \ref{theo:maintheodiffd2punif}, let us conclude this subsection showing that if a Banach space $X$ satisfies that $B_X$ contains slices (resp. non-empty relatively weakly open subsets, convex combinations of slices) of diameter $\leq \delta_0$, then so does the unit ball of $X_\mathcal U$ for any free ultrafilter $\mathcal U$ over any infinite set $I$.

\begin{proposition}\label{prop:smallultrapower}
Let $X$ be a Banach space, $I$ be an infinite set and $\mathcal U$ be a free ultrafilter over $I$. Let $\delta_0>0$. Then:
\begin{enumerate}
\item If $B_X$ contains a slice of diameter $\leq \delta_0$, then $B_{X_\mathcal U}$ contains a slice of diameter $\leq \delta_0$.
\item If $B_X$ contains a non-empty relatively weakly open subset of diameter $\leq \delta_0$, then $B_{X_\mathcal U}$ contains a non-empty relatively weakly open subset of diameter $\leq \delta_0$.
\item If $B_X$ contains a convex combination of slices of diameter $\leq \delta_0$, then $B_{X_\mathcal U}$ contains a convex combination of slices of diameter $\leq \delta_0$.
\end{enumerate}
\end{proposition}

\begin{proof}
Let us prove (2), being the proof of (3) completely similar and the proof of (1) a particular case of (2) (just taking $n=1$).

So assume that there exists $x\in B_X$ and a relatively weakly open subset $W\subseteq B_X$ with $x\in W$ and $\diam(W)\leq \delta_0$. Since slices of $B_X$ are a subbasis of the weak topology of $B_X$ we can assume with no loss of generality that
$$W:=\bigcap\limits_{j=1}^n S(B_X,f_j,\alpha)$$
where $f_1,\ldots, f_n\in S_{X^*}$ and $\alpha>0$ (observe that a common $\alpha>0$ can be taken with no loss of generality thanks to \cite[Lemma 1.4]{ivakhno04}).

Let us construct a weakly open set in $X_\mathcal U$. In order to do so, given $1\leq j\leq n$  consider the functional $\varphi_j([x_i]):=\lim_\mathcal U f_j(x_i)$, and observe that $\varphi_j\in S_{(X_\mathcal U)^*}$. Moreover, observe that
$$\varphi_j([x]):=\lim_{\mathcal U} f_j(x)>1-\alpha$$
since $x\in W$. This proves that
$$V:=\bigcap\limits_{j=1}^n S(B_{X_\mathcal U},\varphi_i,\alpha)$$
is a non-empty weakly open subset of $B_{X_\mathcal U}$. Let us estimate the diameter of $V$. In order to do so, select $\varepsilon>0$ and pick $[x_i],[y_i]\in V$. Observe that we can assume, up to a choice of a different representative, that $x_i\in B_X$ and $y_i\in B_X$ holds for every $i\in I$ thanks to \cite[Remark 2.4]{mr26uni}. Moreover, given $j\in \{1,\ldots, n\}$, we have that
$$\varphi_j([x_i])=\lim_\mathcal U f(x_i)>1-\alpha,$$
so the set
$$A_j:=\{i\in I: f_j(x_i)>1-\alpha\}$$
belongs to $\mathcal U$. Similarly, the set
$$B_j:=\{i\in I: f_j(y_i)>1-\alpha\}\in\mathcal U.$$
Moreover, define the set
$$C:=\{i\in I: \vert \Vert [x_i]-[y_i]\Vert-\Vert x_i-y_i\Vert\vert<\varepsilon\}\in\mathcal U.$$
Finally, the set $D:=\bigcap\limits_{j=1}^n A_j\cap \bigcap\limits_{j=1}^n B_j\cap C\in\mathcal U$ and, consequently, it is non-empty. Now, given $i\in D$ we have that, for every $j\in\{1,\ldots, n\}$, $f(x_i)>1-\alpha$ and $f(y_i)>1-\alpha$ holds, so $x_i,y_i\in W$. Consequently
$$\Vert x_i-y_i\Vert\leq \diam(W)\leq \delta_0.$$
Finally using that $i\in C$ we have that
$$\Vert [x_i]-[y_i]\Vert\leq \Vert x_i-y_i\Vert+\varepsilon\leq \delta_0+\varepsilon.$$
Since $[x_i],[y_i]\in V$ were arbitrary we conclude that $\diam(V)\leq \delta_0+\varepsilon$ and, since $\varepsilon>0$ was arbitrary, we get $\diam(V)\leq \delta_0$, as desired.
\end{proof}

\section{The set STS}\label{section:STS}

Let us recall in this section the construction of the set $STS$ from \cite{aor88}, which is essential in the proof of Theorem~\ref{theo:maintheodiffslid2punif}. 

Let us denote by $\mathcal T$ the infinite branching tree in its usual order, that is, the set of all finite sequences of positive integers, together with the empty sequence $\emptyset\in\mathcal T$, which denotes the origin of the tree. If $\alpha=(\alpha_1,\ldots, \alpha_n)\in \mathcal T$, $\vert\alpha\vert:=n$ is the length of $\alpha$. There is a natural order in $\mathcal T$: given $\alpha=(\alpha_1,\ldots, \alpha_n),\beta=(\beta_1,\ldots, \beta_m)\in\mathcal T$, then $\alpha\leq \beta$ if $\vert\alpha\vert\leq \vert\beta\vert$ and $\alpha_i=\beta_i$ holds for $1\leq i\leq \vert \alpha\vert$. Moreover, we declare $\emptyset \leq (n)$ holds for every $n\in\mathbb N$. Given $\alpha=(\alpha_1,\ldots, \alpha_n)$ and $\beta=(\beta_1,\ldots, \beta_m)$ we denote the concatenation of $\alpha$ and $\beta$ by $\alpha\wedge \beta:=(\alpha_1,\ldots,\alpha_n,\beta_1,\ldots,\beta_m)$, and if $k\in\mathbb N$, we denote by $\alpha\wedge k:=\alpha\wedge (k)$ for simplicity. For any $\alpha\in\mathcal T$ we denote by $S(\alpha)$ the set of all inmediate sucessors of $\alpha$, that is, the set
$$S(\alpha):=\{\alpha\wedge n: n\in\mathbb N\}.$$
Now it is time to introduce the desired set

\begin{definition}\label{defi:STS}
$STS$ is the set of all $x\in c_0(\mathcal T)$ satisfying:
\begin{enumerate}
    \item $x\geq 0\ \forall x\in STS$,
    \item $x(\emptyset)=1$ and,
    \item $\sum_{\beta\in S(\alpha)} x(\beta)\leq x(\alpha)$.
\end{enumerate}
\end{definition}

As we have already mentioned, the construction of the set $STS$ goes back to \cite{aor88}, where the set $STS$ is presented as an example of a subset of $c_0$ which fails the \textit{point of continuity property (PCP)} but enjoys the \textit{convex point of continuity property (CPCP)}. Many (impressive) properties of the set $STS$ are exhibited in \cite[Theorem 1.1]{aor88}. 

Since all slices of $STS$ have diameter $1=\diam(STS)$ and, at the same time, the set $STS$ contains non-empty relatively weakly open subsets of arbitrarily small diameter (a direct consequence of the CPCP), the set $STS$ was used in \cite[Theorem 2.4]{blr15eje1} in order to prove that every Banach space containing $c_0$ admits an equivalent renorming with the slice-D2P but such that the unit ball contains non-empty relatively weakly open subsets of arbitrarily small diameter (in particular, solving the open question whether the slice-D2P and the D2P are equivalent). This idea was also exploited in \cite[Theorem 1.3]{taller} in order to construct an equivalent renorming of $L_\infty([0,1])$ such that the unit ball contains non-empty relatively weakly open subsets of arbitrarily small diameter but such that the set of Daugavet and super $\Delta$-points is as big as possible.

Going back to the set $STS$, let us distinguish some special elements inside: given $\alpha\in\mathcal T$ we denote by $b_\alpha:=\sum_{\beta\leq \alpha} e_\alpha$. It follows that $STS=\overline{\conv}(b_\alpha: \alpha\in\mathcal T\}$. Indeed, let us generalise the above fact in the following proposition, which will be the key to proving the uniform slice-D2P in Theorem~\ref{theo:maintheodiffslid2punif}.

\begin{proposition}\label{prop:aprouniSTS}
Let $x\in STS$ and $n\geq 2$. Then there exists $z\in \conv_n(\{b_\alpha: \alpha\in\mathcal T\})$ such that
$$\Vert z-x\Vert\leq \frac{3}{n}.$$
\end{proposition}

\begin{proof}
Since $x\in c_0(\mathcal T)$ the set
$$F:=\left\{\alpha\in \mathcal T: x(t)\geq \frac{1}{n} \right\}$$
is finite. Let $L:=\max\{\vert \alpha\vert: \alpha\in F\}$. 

Let us construct, for every $\alpha\in F$, an element $\lambda_\alpha\in [0,1]$ of the form $\lambda_\alpha=\frac{k_\alpha}{n}$, where $k_\alpha\in\mathbb N\cup\{0\}$, such that
\begin{equation}\label{eq:aproSTS}
0\leq x(\alpha)-\sum_{F\ni \beta\geq \alpha}\lambda_\alpha<\frac{1}{n}\ \forall \alpha\in F.\end{equation}
The construction will be done by induction on $\vert \alpha\vert$.

Given $\alpha\in F$ with $\vert \alpha\vert=L$, then $\alpha\wedge n\notin F\ \forall n\in\mathbb N$ since $L$ is the maximal level possible for an element in $F$. 

Now define $\lambda_\alpha:=\frac{k_\alpha}{n}$, where $k_\alpha\in\mathbb N\cup\{0\}$ is such that $\frac{k_\alpha}{n}\leq x(\alpha)<\frac{k_\alpha+1}{n}$. Clearly $\lambda_\alpha$ satisfies the requirements.

Now assume that $0\leq k_0<L$ and that $\lambda_\alpha=\frac{k_\alpha}{n}$ has been constructed for every $\alpha\in F$ with $\vert \alpha\vert>k_0$. Let us construct $\lambda_\alpha$ for every $\alpha\in F$ such that $\vert \alpha\vert=k_0$. Given one such $\alpha$ we have two possibilities:\vspace{0.3cm}

\textbf{Case 1:} $\alpha\wedge n\notin F$ for every $n\in\mathbb N$. In this case we define $\lambda_\alpha$ as in the above step. \vspace{0.3cm}

\textbf{Case 2:} If $\alpha\wedge n_0\in F$ for some $n_0$, then $\lambda_\beta$ has been constructed for every $\beta\in F, \beta>\alpha$. Observe that 
$$\sum_{F\ni \beta>\alpha}\lambda_\beta=\sum_{\beta\in S(\alpha)\cap F}\sum_{F\ni t\geq \beta}\lambda_t\leq \sum_{\beta\in S(\alpha)\cap F}x(\beta)\leq \sum_{\beta\in S(\alpha)}x(\beta)\leq x(\alpha),$$
where the first inequality follows by the inductive step for any successor of $\alpha$ which belongs to $F$, whereas the second inequality is obvious because $x$ is positive and the last one follows by the definition of $STS$.

Now, we have two possibilities:
\begin{itemize}
    \item If $0\leq x(\alpha)-\sum_{F\ni \beta>\alpha}\lambda_\beta<\frac{1}{n}$ define $\lambda_\alpha=0$, and clearly \eqref{eq:aproSTS} holds.
    \item If $\frac{1}{n}\leq x(\alpha)-\sum_{F\ni \beta>\alpha}\lambda_\beta$ then define $\lambda_\alpha:=\frac{k_\alpha}{n}$, where $k_\alpha\in\mathbb N$ is such that
    $$\frac{k_\alpha}{n}\leq x(\alpha)-\sum_{F\ni \beta>\alpha}\lambda_\beta<\frac{k_\alpha+1}{n}.$$
\end{itemize}
It remains to prove \eqref{eq:aproSTS}. On the one hand
$$x(\alpha)-\sum_{F\ni \beta\geq \alpha}\lambda_\beta=x(\alpha)-\sum_{F\ni \beta> \alpha}\lambda_\beta-\frac{k_\alpha}{n}\geq 0.$$
On the other hand
$$x(\alpha)-\sum_{F\ni \beta\geq \alpha}\lambda_\beta=x(\alpha)-\sum_{F\ni \beta>\alpha}\lambda_\beta-\frac{k_\alpha}{n}=x(\alpha)-\sum_{F\ni \beta>\alpha}\lambda_\beta-\frac{k_\alpha+1}{n}+\frac{1}{n}<\frac{1}{n}.$$
This completes the inductive process.

Now set $y:=\sum_{\alpha\in F}\lambda_\alpha b_\alpha$. Let us first prove that $\Vert y-x\Vert=\sup_{\alpha\in\mathcal T}\vert y(\alpha)-x(\alpha)\vert<\frac{1}{n}$. Given $\alpha\in \mathcal T$ we have two possibilities:
\begin{enumerate}
    \item If $\alpha\in F$ then we have
    $$x(\alpha)-y(\alpha)=x(\alpha)-\sum_{\beta\in F}\lambda_\beta b_\beta(\alpha)=x(\alpha)-\sum_{F\ni \beta\geq \alpha}\lambda_\beta$$
    which belongs to $[0,\frac{1}{n})$ by construction, so $\vert x(\alpha)-y(\alpha)\vert<\frac{1}{n}$. 

    \item If $\alpha\notin F$ then $y(\alpha)=0$. Consequently
    $$\vert x(\alpha)-y(\alpha)\vert=x(\alpha)<\frac{1}{n}$$
    as $F:=\{t\in \mathcal T: x(t)\geq \frac{1}{n}\}$. 
\end{enumerate}
The above discussion proves that $\Vert y-x\Vert\leq \frac{1}{n}$.

Now recall that for every $\alpha\in F$ we have $\lambda_\alpha=\frac{k_\alpha}{n}$, for some $k_\alpha\in\mathbb N\cup\{0\}$. Now set
$$P:=\{\alpha\in F: k_\alpha>0\}.$$
By an application of \eqref{eq:aproSTS} to $\alpha=\emptyset$ we get
$$0<1-\sum_{\alpha\in F} \lambda_\alpha=1-\sum_{\alpha\in P}\lambda_\alpha<\frac{1}{n}.$$
Now
$$1\geq \sum_{\alpha\in P}\lambda_\alpha=\sum_{\alpha\in P}\frac{k_\alpha}{n}\geq \frac{1}{n}\sum_{\alpha\in P}1=\frac{\vert P\vert}{n},$$
so $\vert P\vert\leq n$. Now define $\Lambda:=\sum_{\alpha\in F}\lambda_\alpha$, which satisfies $1-\frac{1}{n}<\Lambda\leq 1$. Finally set $z:=\frac{y}{\Lambda}=\sum_{\alpha\in P} \frac{\lambda_\alpha}{\Lambda}b_\alpha\in \conv_{\vert P\vert}(\{b_\alpha: \alpha\in\mathcal T\})\subseteq \conv_{n}(\{b_\alpha: \alpha\in\mathcal T\})$. Moreover
\[\begin{split}
\Vert x-z\Vert& \leq \Vert x-y\Vert+\Vert y-z\Vert\leq \frac{1}{n}+\left\vert 1-\frac{1}{\Lambda}\right\vert \Vert y\Vert\\
 & \leq  \frac{1}{n}+\frac{1-\Lambda}{\Lambda}<\frac{1}{n}+\frac{\frac{1}{n}}{1-\frac{1}{n}}\leq \frac{3}{n},
\end{split}\]
which finishes the proof.
\end{proof}

\section{Proof of Theorem~\ref{theo:maintheodiffslid2punif}}\label{section:prooftheoremunislice}

Let $X$ be a Banach space containing $c_0(\mathbb N)=c_0(\mathcal T)$ (this follows since $\mathcal T$ is clearly countable). Up to an application of  Lemma~\ref{lemma:c0inicial} and an isometric identification of $c(\mathbb N)$ and $c(\mathcal T):=c_0(\mathcal T)\oplus \mathbb R\1\subseteq \ell_\infty(\mathcal T)$, we can assume that $X$ contains an isometric copy of $c(\mathcal T)$ and satisfies that, given $x\in B_X$ and $\alpha\in\mathcal T$, there exists $a_\alpha\in\mathbb R$ such that $\vert a_\alpha\vert\leq 1$ and such that
$$x+(1-a_\alpha)e_\alpha\mbox{ and }x-(1+a_\alpha)e_\alpha$$
belong to $B_X$ for every $\alpha\in\mathcal T$. Let $\varepsilon>0$ and consider the equivalent norm $\Vert\cdot\Vert_\varepsilon$ on $X$ whose unit ball is 
$$B_\varepsilon:=\overline{\co}(A\cup -A\cup ((1-\varepsilon)B_X+\varepsilon B_{c_0}),$$
where $A:=2STS-\1$. Observe that, since $STS=\overline{\co}(\{b_\alpha: \alpha\in\mathcal T\})$, the above renorming is the renorming considered in \cite[Theorem 2.4]{blr15eje1}, which is proved there to enjoy the slice-D2P but such that the unit ball contains non-empty relatively weakly open subsets of arbitrarily small diameter. Our aim is to prove that $(X,\Vert\cdot\Vert_\varepsilon)$ has the uniform slice-D2P. In order to do so, our approach will be different from that of \cite[Theorem 2.4]{blr15eje1}, where it was used that any slice of $B_\varepsilon$ must intersect $A\cup -A\cup ((1-\varepsilon)B_X+\varepsilon B_{c_0})$. We will look for an application of Theorem~\ref{theo:charunislice}. 

Call $B:=((1-\varepsilon)B_X+\varepsilon B_{c_0})$. Since $\frac{A-A}{2}\subseteq B$, we obtain by \cite[Lemma 2.4]{blr15eje2} that 
$$\conv(A\cup -A\cup B)=\conv(A\cup B)\cup \conv(-A\cup B).$$
As a consequence, we obtain that $\conv(A\cup B)\cup \conv(-A\cup B)$ is dense in $B_\varepsilon$. Since $B_{c_{00}}$ is dense in $B_{c_0}$ we then infer that
$$\conv(A\cup ((1-\varepsilon)B_X+\varepsilon B_{c_{00}}))\cup \conv(-A\cup ((1-\varepsilon)B_X+\varepsilon B_{c_{00}}))$$
is dense in $B_\varepsilon$. 

We will make use of this fact. In order to prove the uniform slice-D2P, select any $\delta>0$ and any $n\in\mathbb N$ with $n\geq 2$. Select any $w\in \conv(A\cup ((1-\varepsilon)B_X+\varepsilon B_{c_{00}}))\cup \conv(-A\cup ((1-\varepsilon)B_X+\varepsilon B_{c_{00}}))$, which is dense in $B_\varepsilon$. With no loss of generality we can assume that $w\in \conv(A\cup ((1-\varepsilon)B_X+\varepsilon B_{c_{00}}))$, so we can write
$$w=\lambda (2\varphi-\1)+(1-\lambda)((1-\varepsilon)x+\varepsilon y),$$
where $\lambda\in [0,1]$, $\varphi\in STS$, $x\in B_X$ and $y\in B_{c_{00}}$. In virtue of Proposition~\ref{prop:aprouniSTS} we can find $\lambda_1,\ldots, \lambda_n\in [0,1]$ such that $\sum_{j=1}^n \lambda_j=1$ and $\alpha_1,\ldots, \alpha_n\in \mathcal T$ such that
$$\left\Vert \varphi-\sum_{j=1}^n \lambda_j b_{\alpha_j} \right\Vert<\frac{3}{n}.$$
Now, given $1\leq j\leq n$, define $w_j:=\lambda(2b_{\alpha_j}-\1)+(1-\lambda)((1-\varepsilon)x+\varepsilon y)$, and observe that
\[
\begin{split}
\left\Vert w-\sum_{j=1}^n \lambda_j w_j\right\Vert & =\left\Vert \lambda\left(2\left(\varphi-\sum_{j=1}^n \lambda_j b_{\alpha_j} \right)-\left(\1-\sum_{j=1}^n \lambda_j\1 \right)\right) \right. \\ 
& \left. +(1-\lambda)\left((1-\varepsilon)\left( x-\sum_{j=1}^n \lambda_j x\right)+\varepsilon\left(y-\sum_{j=1}^n \lambda_j y \right) \right) \right\Vert\\
& = 2 \lambda\left\Vert \varphi-\sum_{j=1}^n \lambda_j b_{\alpha_j}\right\Vert<\frac{6}{n}.
\end{split}\]
Since $\Vert z\Vert_\infty\leq \Vert z\Vert_\varepsilon\leq \frac{1}{1-\varepsilon}\Vert z\Vert$ holds for every $z\in X$ we get that
\begin{equation}\label{eq:aproiniteounislice}
\left \Vert w-\sum_{j=1}^n \lambda_j w_j\right\Vert_\varepsilon<\frac{6}{1-\varepsilon}\frac{1}{n}.
\end{equation}
Taking into account the above inequality we will focus on approximating each $w_j$ by convex combinations of elements in the set
$$\left\{\frac{u+v}{2}: u,v\in B_\varepsilon, \Vert u-v\Vert_\varepsilon=2 \right\}.$$
Since $y$ is finitely supported we can find $m_0\in\mathbb N$ large enough to guarantee $y(\alpha_j\wedge m_0+k)=0$ holds for every $1\leq j\leq n$ and every $k\in\mathbb N$.

Now select $j\in \{1,\ldots, n\}$ and let us approximate $w_j$. In order to do so define, given $1\leq i\leq n$, the elements
\[\begin{split}u_i^j& :=\lambda(2b_{\alpha_j\wedge m_0+i}-\1)\\& +(1-\lambda)((1-\varepsilon)(x+(1-a_{\alpha_j\wedge m_0+i})e_{\alpha_j\wedge m_0+i})+\varepsilon (y+e_{\alpha_j\wedge m_0+i}))\end{split}\]
and
$$v_i^j:=\lambda(2b_{\alpha_j}-\1)+(1-\lambda)((1-\varepsilon)(x-(1+a_{\alpha_j\wedge m_0+i})e_{\alpha_j\wedge m_0+i})+\varepsilon (y-e_{\alpha_j\wedge m_0+i})).$$
Observe that both elements belong to $B_\varepsilon$ for every $1\leq i\leq n$. Moreover, since $b_{\alpha_j\wedge m_0+i}-b_{\alpha_j}=2e_{\alpha_j\wedge m_0+i}$ we have that
$$u_{i}^j-v_{i}^j=\lambda (2e_{\alpha_j\wedge m_0+i})+(1-\lambda)((1-\varepsilon) 2e_{\alpha_j\wedge m_0+i}+\varepsilon 2e_{\alpha_j\wedge m_0+i})=2e_{\alpha_j\wedge m_0+i}.$$
Thus
$$\Vert u_i^j-v_i^j\Vert_\varepsilon=2\Vert e_{\alpha_j\wedge m_0+i}\Vert_\varepsilon\geq 2\Vert e_{\alpha_j\wedge m_0+i}\Vert_\infty=2.$$
Now let us estimate $w_j-\frac{1}{n}\sum_{i=1}^n \frac{u_i^j-v_i^j}{2}$. To do so, given $1\leq i,j\leq n$, using that $b_{\alpha_j\wedge m_0+i}+b_j=2 b_{\alpha_j}+e_{\alpha_j\wedge m_0+i}$ we have
\[\begin{split}\frac{u_i^j+v_i^j}{2}& =\lambda(2 b_{\alpha_j}-\1)+(1-\lambda)((1-\varepsilon)x+\varepsilon y)\\ 
& +\lambda \frac{e_{\alpha_j\wedge m_0+i}}{2}+(1-\lambda)((1-\varepsilon)(-2a_{\alpha_j\wedge m_0+i})e_{\alpha_j\wedge m_0+i})\\
& = w_j + \lambda \frac{e_{\alpha_j\wedge m_0+i}}{2}+(1-\lambda)((1-\varepsilon)(-2a_{\alpha_j\wedge m_0+i})e_{\alpha_j\wedge m_0+i})
\end{split}\]
Consequently
\[\begin{split}
w_j-\frac{1}{n}\sum_{i=1}^n \frac{u_i^j+v_i^j}{2}& =\frac{1}{n}\sum_{i=1}^n \left(\lambda \frac{e_{\alpha_j\wedge m_0+i}}{2}+(1-\lambda)((1-\varepsilon)(-2a_{\alpha_j\wedge m_0+i})e_{\alpha_j\wedge m_0+i}) \right)\\
& = \frac{1}{n} \sum_{i=1}^n \left(\frac{\lambda}{2}-2(1-\lambda)(1-\varepsilon)a_{\alpha_j\wedge m_0+i}\right)e_{\alpha_j\wedge m_0+i}
\end{split}\]
Thus, taking into account that $\vert a_{\alpha_j\wedge m_0+i}\vert\leq 1$ holds for every $i, j$, we get
\[\begin{split}
\left\Vert w_j-\frac{1}{n}\sum_{i=1}^n \frac{u_i^j+v_i^j}{2}\right\Vert & =\frac{1}{n} \max\limits_{1\leq j\leq n}\left\{\left\vert \frac{\lambda}{2}-2(1-\lambda)(1-\varepsilon)a_{\alpha_j\wedge m_0+i}\right\vert \right\} \\
& \leq \frac{1}{n} \left(\frac{\lambda}{2}+2(1-\lambda)(1-\varepsilon) \right)\leq \frac{3}{n}.
\end{split}\]
Taking into account the equivalence constants we get
$$\left\Vert w_j-\frac{1}{n}\sum_{i=1}^n \frac{u_i^j+v_i^j}{2}\right\Vert_\varepsilon\leq \frac{3}{1-\varepsilon}\frac{1}{n}$$
Now define $z:=\sum_{i=1}^n \sum_{j=1}^n \frac{\lambda_i}{n}\frac{ u_i^j+v_i^j}{2}$. We have proved that
$$z\in \conv_{n^2}\left(\left\{\frac{u+v}{2}: u,v\in B_\varepsilon, \Vert u-v\Vert_\varepsilon>2-\delta\right\} \right).$$
Moreover, we have
\[\begin{split}
\Vert w-z\Vert& \leq \left\Vert w-\sum_{j=1}^n w_j\right\Vert_\varepsilon+\sum_{j=1}^n \lambda_j \left\Vert w_j-\frac{1}{n}\sum_{i=1}^n \frac{u_i^j+v_i^j}{2}
\right\Vert_\varepsilon \\
& < \frac{6}{1-\varepsilon}\frac{1}{n}+\sum_{j=1}^n \lambda_j \frac{1}{1-\varepsilon}\frac{3}{n}=\frac{9}{1-\varepsilon}\frac{1}{n}. 
\end{split}\]
Consequently $d(w, S_{n^2}^{2-\delta}(X))\leq \frac{9}{1-\varepsilon}\frac{1}{n}$. Observe that if $w\in \conv(-A\cup ((1-\varepsilon)B_X+\varepsilon B_{c_{00}}))$ we would get the same conclussion because $-w\in \conv(A\cup ((1-\varepsilon)B_X+\varepsilon B_{c_{00}}))$ and the set
$$\conv_{n^2}\left(\left\{\frac{u+v}{2}: u,v\in B_\varepsilon, \Vert u-v\Vert_\varepsilon>2-\delta\right\} \right)$$
is clearly symmetric.

Consequently, we conclude that given any 
$$w\in \conv(A\cup ((1-\varepsilon)B_X+\varepsilon B_{c_{00}})\cup \conv(-A\cup ((1-\varepsilon)B_X+\varepsilon B_{c_{00}})$$
we get $d(w, S_{n^2}^{2-\delta}(X))\leq \frac{9}{1-\varepsilon}\frac{1}{n}$, and by the density of the above set we conclude that 
$$C_{n^2}^{2-\delta}=\sup_{w\in B_\varepsilon}d(w, S_{n^2}^{2-\delta}(X))\leq \frac{9}{1-\varepsilon}\frac{1}{n}.$$
Since the sequence $\{C_n^{2-\delta}\}_{n\in\mathbb N}$ is decreasing in $n$ we get that $\lim_n C_n^{2-\delta}=0$ for every $\delta>0$, and thus $X$ has the uniform slice-D2P in virtue of Theorem~\ref{theo:charunislice}, which completes the proof of Theorem~\ref{theo:maintheodiffslid2punif}.

\begin{remark}
Observe that Theorem~\ref{theo:maintheodiffslid2punif} together with Proposition~\ref{prop:smallultrapower} implies that there are Banach spaces $X$ such that, for any free ultrafilter over $\mathbb N$, it follows that $X_\mathcal U$ has the slice-D2P but its unit ball contains non-empty relatively weakly open subsets of arbitrarily small diameter.
\end{remark}

\section{Difference between the uniform D2P and the uniform SD2P}\label{section:prelisegureno}

In this short section we will present some specific material that we will need for the proof of Theorem~\ref{theo:maintheodiffd2punif} in the next section. First of all, let us provide a description of a set which is employed in the renorming technique used in \cite[Theorem 2.5]{blr15eje2}.

Pick a nonincreasing null sequence $\{\varepsilon_n\}$ in $\mathbb R^+$. We construct an increasing sequence of closed, bounded and convex subsets $\{K_n\}$ in $c_0$ and a sequence $\{g_n\}$ in $c_0$ as
follows: First define $K_1=\{e_1\}$, $g_1=e_1$ and $K_2=\conv(e_1, e_1+e_2)$. Choose $l_2>1$ and an $\varepsilon_2$-net $g_2,\ldots ,g_{l_2}\in K_2$ in $K_2$. Assume that $n\geq 2$ and that $m_n,\ l_n
,\ K_n$ and $\{g_1,\ldots ,g_{l_n}\}$ have been constructed, with
$K_n\subseteq B_{\spann\{e_1,\ldots ,e_{m_n}\}}$ and $g_i\in K_n$ for every
$1\leq i\leq l_n$. Define $K_{n+1}$ as $$K_{n+1}=\conv(K_n\cup
\{g_i+e_{m_n+i}:1\leq i\leq l_n\}).$$ Consider $m_{n+1}=m_n+l_n$ and
choose $\{g_{l_{n}+1},\ldots ,g_{l_{n+1}}\}\in K_{n+1}$ so that
$\{g_1 ,\ldots ,g_{l_{n+1}}\}$ is an $\varepsilon_{n+1}$-net in
$K_{n+1}$. Finally we define $K_0=\overline{\cup_n K_n}$. Then it
follows that $K_0$ is a non-empty closed, bounded and convex subset
of $c_0$ such that $x(n)\geq 0$ for every $n\in \natu$ and $\Vert
x\Vert_{\infty}=1$ for every $x\in K_0$ and so $\diam(K_0)\leq 1$.

Now, fixed $i$, we have from the construction that
$\{g_i+e_{m_n+i}\}_n$ is a sequence in $K_0$ which is weakly convergent to
$g_i$ and $\Vert (g_i-e_{m_n+i})-g_i\Vert=\Vert e_{m_n+i}\Vert=1$
holds for every $n$. Then $\diam(K_0)=1$. We will freely use the set $K_0$ and the above construction throughout the rest of the paper. Observe that, from the
above construction, it follows that
$$K_0=\overline{\{g_i:i\in\natu\}}^{w}=\overline{\{g_i:i\in\natu\}}.$$
As before, the set $K_0$ above appears in the paper \cite{aor88} as an example of a subset of $c_0$ which fails the CPCP but such that $K_0$ contains convex combinations of slices of arbitrarily small diameter \cite[Theorem 1.2]{aor88}. Similarly to the case of the set STS, the property of $K_0$ that contains convex combinations of slices of arbitrarily small diameter and at the same time every non-empty relatively weakly open subset of $K_0$ has diameter $1=\diam(K_0)$ was used in \cite[Theorem 2.5]{blr15eje2} to provide examples of Banach spaces with the D2P but whose unit ball contains arbitrarily small convex combinations of slices. Moreover, the above set $K_0$ was used later to construct several counterexamples. For instance, a renorming of $C([0,1])$ involving $K_0$ was done in \cite[Section 4.6]{MPR} in order to provide an example of a super Daugavet point which is not a ccs $\Delta-$point. More recently, the set $K_0$ was used in \cite{lmr26} in order to get an equivalent renorming of $c_0\oplus_\infty\mathbb R$ with the property that every non-empty relatively weakly open subset of the unit ball has radius $1$ but such that the unit ball contains slices of diameter close to $1$ too. 

Let us conclude with the next result which, having an immediate proof, provides us an interesting criterion to determine the D2P in a ultrapower space, where we have not a good description of weak open sets.

\begin{proposition}\label{prop:condisufiD2P}
Let $X$ be a Banach space. Assume that for every $x\in B_X$ there are two sequences $(x_n),(y_n)$ in $B_X$ such that
\begin{enumerate}
    \item $(x_n-x)$ and $(y_n-y)$ are weakly null and,
    \item $\Vert x_n-y_n\Vert=2$ holds for every $n\in\mathbb N$.
\end{enumerate}
Then $X$ has the D2P.
\end{proposition}

\section{Proof of Theorem~\ref{theo:maintheodiffd2punif}}\label{section:prooftheoremunid2p}

Let $X$ be a Banach space containing $c_0$. Up to an application of  Lemma~\ref{lemma:c0inicial} we can assume that $X$ contains an isometric copy of $c$ and satisfies that, given $x\in B_X$ and $n\in\mathbb N$, there exists $a_n\in\mathbb R$ such that $\vert a_n\vert\leq 1$ and such that
$$x+(1-a_n)e_n\mbox{ and }x-(1+a_n)e_n$$
belong to $B_X$ for every $n\in\mathbb N$. Let $\varepsilon>0$ and consider the equivalent norm $\Vert\cdot\Vert_\varepsilon$ on $X$ whose unit ball is 
$$B_\varepsilon:=\overline{\co}((2K_0-\1)\cup -(2K_0-\1)\cup ((1-\varepsilon)B_X+\varepsilon B_{c_0})),$$
where $K_0$ is the set constructed in the previous section. In \cite[Theorem 2.5]{blr15eje2} it is proved that $(X,\Vert\cdot\Vert_\varepsilon)$ satisfies the D2P and that the unit ball contains convex combinations of slices of arbitrarily small diameter. Let us prove that $X_\varepsilon:=(X,\Vert\cdot\Vert_\varepsilon)$ has the uniform D2P.

In order to do so select any infinite set $I$ and any countably incomplete ultrafilter $\mathcal U$ over $I$. Let us prove that $(X_\varepsilon)_\mathcal U$ has the D2P. Call $A:=(2K_0-\1)$ and $B:=((1-\varepsilon)B_X+\varepsilon B_{c_0})$. Since $\frac{A-A}{2}\subseteq B$ we have by \cite[Lemma 2.4]{blr15eje2} that
$$\conv(A\cup -A\cup B)=\conv(A\cup B)\cup \conv(-A\cup B).$$
Since $\conv(A\cup -A\cup B)$ is dense in $B_\varepsilon$ we have by Proposition~\ref{prop:ultracierre} the equality
$$B_{(X_\varepsilon)_\mathcal U}=(B_{X_\varepsilon})_\mathcal U=\conv(A\cup -A\cup B)_\mathcal U=\left(\conv(A\cup B)\cup \conv(-A\cup B) \right)_\mathcal U.$$
Thanks to Proposition~\ref{prop:ultraunion} we get that 
$$B_{(X_\varepsilon)_\mathcal U}=\conv(A\cup B)_\mathcal U \cup \conv(-A\cup B)_\mathcal U.$$
In order to prove that $(X_\varepsilon)_\mathcal U$ has the D2P let us apply Proposition~\ref{prop:condisufiD2P}. Let $[u_i]\in \conv(A\cup B)_\mathcal U\cup \conv(-A\cup B)_\mathcal U$. Let us assume with no loss of generality that $[u_i]\in \conv(A\cup B)_\mathcal U$. Observe that since 
$$\conv(A\cup ((1-\varepsilon)B_X+\varepsilon B_{c_{00}}))$$
is dense in $\conv(A\cup B)$, a new application of Proposition~\ref{prop:ultracierre} reveals that
$$[u_i]\in \conv(A\cup B)_\mathcal U=\conv(A\cup ((1-\varepsilon)B_X+\varepsilon B_{c_{00}}))_\mathcal U,$$
so we can assume that, for every $i\in I$, we have
$$u_i=\lambda_i (2g_{p_i}-\1)+(1-\lambda_i)((1-\varepsilon)x_i+\varepsilon y_i),$$
where $\lambda_i\in [0,1]$, $p_i\in\mathbb N$, $x_i\in B_X$ and $y_i\in B_{c_{00}}$. Since $y_i$ has finite support we can find $q_i\in \mathbb N$ such that 
$$y_i(m_n+p_i)=0\ \forall n\geq q_i.$$
In order to apply Proposition~\ref{prop:condisufiD2P} let us consider, given $n\in\mathbb N$ and $i\in I$ the elements
\[\begin{split} u_i^n:=\lambda_i (2 (g_{p_i}+e_{m_{n+q_i}+p_i})-\1)& +(1-\lambda_i)((1-\varepsilon)(x+(1-a_{m_{n+q_i}+p_i})e_{m_{n+q_i}+p_i})\\
& +\varepsilon(y+e_{m_{n+q_i}+p_i})\end{split}\]
and
\[\begin{split} v_i^n:=\lambda_i (2 g_{p_i}-\1)& +(1-\lambda_i)((1-\varepsilon)(x-(1+a_{m_{n+q_i}+p_i})e_{m_{n+q_i}+p_i})\\
& +\varepsilon(y-e_{m_{n+q_i}+p_i}).\end{split}\]
Observe that $u_i^n, v_i^n\in B_\varepsilon$ holds for every $n\in\mathbb N$ and every $i\in I$. Let us prove that the sequences $([u_i^n])_n,([v_i^n])_n$ satisfy the hypothesis of Proposition~\ref{prop:condisufiD2P}. Given $i\in I$ and $n\in\mathbb N$ we have
$$u_i^n-v_i^n=\lambda_i 2 e_{m_{n+q_i}+p_i}+(1-\lambda_i)((1-\varepsilon)2e_{m_{n+q_i}+p_i}+\varepsilon2e_{m_{n+q_i}+p_i})=2e_{m_{n+q_i}+p_i}.$$
Thus $\Vert u_i^n-v_i^n\Vert_\varepsilon\geq \Vert u_i^n-v_i^n\Vert=2\Vert e_{m_{n+q_i}+p_i}\Vert=2$. The arbitrariness of $i\in I$ implies
$$\Vert [u_i^n]-[v_i^n]\Vert=\lim_\mathcal U\Vert u_i^n-v_i^n\Vert_\varepsilon=2.$$
So in order to apply Proposition~\ref{prop:condisufiD2P} it remains to prove that both $([u_i^n]-[u_i])_n$ and $([v_i^n]-[u_i])_n$ are weakly null. Let us begin with $([u_i^n]-[u_i])_n$. Given $i\in I$ and $n\in\mathbb N$ we have
\[\begin{split}
u_i^n-u_i& =\lambda_i 2e_{m_{n+q_i}+p_i}+(1-\lambda_i)((1-\varepsilon)(1-a_{m_{n+q_i}+p_i})e_{m_{n+q_i}+p_i}+\varepsilon e_{m_{n+q_i}+p_i})\\
& =(2\lambda_i+(1-\lambda_i)(1-(1-\varepsilon) a_{m_{n+q_i}+p_i})e_{m_{n+q_i}+p_i}\\
& = (1+\lambda_i-(1-\varepsilon)(1-\lambda_i)a_{m_{n+q_i}+p_i})e_{m_{n+q_i}+p_i}.
\end{split}\]
In view of the above, given $c_1,\ldots, c_t\in\mathbb R$ we have that
\[\begin{split}\left\Vert \sum_{j=1}^t c_j(u_i^j-u_i) \right\Vert& =\max_{1\leq j\leq t} \vert c_j\vert \left\vert (1+\lambda_i-(1-\varepsilon)(1-\lambda_i)a_{m_{j+q_i}+p_i}) \right\vert\\
& \leq  \max_{1\leq j\leq t} \vert c_j\vert (1+1+(1-\varepsilon)(1-\lambda)\vert a_{m_{j+q_i}+p_i}\vert)\\
& \leq 3 \max_{1\leq j\leq t}\vert c_j\vert.
\end{split}\]
Taking into account the equivalence of norms we get that
$$\left\Vert \sum_{j=1}^t c_j(u_i^j-u_i) \right\Vert_\varepsilon\leq \frac{3}{1-\varepsilon} \max_{1\leq j\leq t}\vert c_j\vert.$$
Since $i\in I$ was arbitrary we conclude that
$$\left\Vert \sum_{j=1}^t c_j([u_i^j]-[u_i]) \right\Vert\leq \frac{3}{1-\varepsilon} \max\limits_{1\leq j\leq t} \vert c_j\vert.$$

Since $c_1,\ldots, c_t\in\mathbb R$ were arbitrary we conclude that the operator
$$\begin{array}{ccc}
 \phi: c_0 & \longrightarrow  & (X_\varepsilon)_\mathcal U  \\
  e_n & \longmapsto   & [u_i^n]-[u_i]
\end{array}$$
is a linear and bounded operator with $\Vert \phi\Vert\leq \frac{3}{1-\varepsilon}$. Since $\phi$ is continuous it is $w-w$ continuous. Hence $(\phi(e_n))_n=([u_i^n]-[u_i])_n$ is weakly null.

Let us now move to the case of $([v_i^n]-[u_i])_n$. Given $i\in I$ and $n\in\mathbb N$ observe that
\[\begin{split}
v_i^n-u_i& =(1-\lambda_i)((1-\varepsilon)(-1)(1+a_{m_{n+q_i}+p_i})e_{m_{n+q_i}+p_i}-\varepsilon e_{m_{n+q_i}+p_i})\\
& =-(1-\lambda_i)((1-\varepsilon)(1+a_{m_{n+q_i}+p_i})+\varepsilon)e_{m_{n+q_i}+p_i}\\
& = -(1-\lambda_i)(1+(1-\varepsilon)a_{m_{n+q_i}+p_i})e_{m_{n+q_i}+p_i}
\end{split}\]
In view of the above, given $c_1,\ldots, c_t\in\mathbb R$ we have that
\[\begin{split}\left\Vert \sum_{j=1}^t c_j(v_i^j-u_i) \right\Vert& =\max_{1\leq j\leq t} \vert c_j\vert \left\vert(1-\lambda_i)(1+(1-\varepsilon)a_{m_{j+q_i}+p_i}) \right\vert\\
& \leq  \max_{1\leq j\leq t} \vert c_j\vert (1+(1-\varepsilon)\vert a_{m_{j+q_i}+p_i}\vert)\\
& \leq 2 \max_{1\leq j\leq t}\vert c_j\vert.
\end{split}\]
With a similar argument to the above we obtain that the operator 
$$\begin{array}{ccc}
 \psi: c_0 & \longrightarrow  & (X_\varepsilon)_\mathcal U  \\
  e_n & \longmapsto   & [v_i^n]-[u_i]
\end{array}$$
is bounded with $\Vert\psi\Vert\leq \frac{2}{1-\varepsilon}$, and then we conclude that $([v_i^n]-[u_i])_n$ is weakly null.

The arbitrariness of $[u_i]\in B_{(X_\varepsilon)_\mathcal U}$ implies by Proposition~\ref{prop:condisufiD2P} that $(X_\varepsilon)_\mathcal U$ has the D2P, as desired.

Let us conclude with a final remark.

\begin{remark}\begin{enumerate}
\item Observe that Theorem~\ref{theo:maintheodiffd2punif} together with Proposition~\ref{prop:smallultrapower} imply that there are Banach spaces $X$ such that, for any free ultrafilter over $\mathbb N$, it follows that $X_\mathcal U$ has the D2P but its unit ball contains convex combinations of slices of arbitrarily small diameter.
\item In the proof of Theorem~\ref{theo:maintheodiffd2punif} we have proved not only that $X$ has the uniform D2P (i.e. $X_\mathcal U$ has the D2P for every free ultrafilter $\mathcal U$ over $\mathbb N$) but also that $X_\mathcal U$ has the D2P for any countably incomplete ultrafilter over any infinite set $I$. This result is not surprising, as it is known that the uniform SD2P has this behaviour \cite[Corollary 3.8]{mr26uni} and, with a similar proof, one can similarly prove it for the uniform slice-D2P. For the D2P, however, the situation is different since it is not evident that, in general, a Banach space $X$ has the uniform D2P if and only if $X_\mathcal U$ has the D2P for any countably incomplete ultrafilter $\mathcal U$ over any infinite set $I$ (even though we believe that such equivalence holds true).
\end{enumerate}
\end{remark}

\section{Concluding  remarks}\label{section:remarks}

As we pointed out in the introduction, it was proved in \cite[Remark 3.10]{mr26uni} that the uniform D2P and the uniform SD2P are different properties just considering $X\oplus_p Y$ with $1<p<\infty$ and $X,Y$ being two Banach spaces with the uniform SD2P. The above result is based on the fact that the isometric identification $X_\mathcal U\oplus_p Y_\mathcal U=(X\oplus_p Y)_\mathcal U$ implies that the uniform D2P is inherited by finite $\ell_p$-sums. At this point, the following question makes sense.

\begin{question}\label{question:uniformd2plpsums}
Let $\{X_i: i\in I\}$ be a family of Banach spaces with the uniform D2P and let $1\leq p<\infty$. Is it true that $\ell_p(I,X_i)$ has the uniform D2P?
\end{question}

Our interest in the above question is that, if the answer were affirmative,  we could get examples of Banach spaces with the uniform D2P but whose unit ball contains convex combinations of slices of arbitrarily small diameter thanks to the following result.

\begin{proposition}\label{prop:ccslicearbipe}
Let $\{X_n: n\in\mathbb N\}$ be a sequence of non-zero Banach spaces and let $1<p<\infty$. Then $X:=\left(\oplus_{n=1}^\infty X_n \right)_p$ contains convex combinations of slices of $B_X$ of arbitrarily small diameter.    
\end{proposition}

\begin{proof}
Let $\varepsilon>0$ and consider a sequence $(\lambda_n)\in S_{\ell_1}$ such that $\lambda_n\geq 0$ and such that $\sum_{n=1}^\infty \lambda_n^p<\varepsilon^p$ (such sequence exists since the the inclusion operator $i:\ell_1\longrightarrow \ell_p$ is not bounded below). We can assume that there exists $n\in\mathbb N$ such that $\lambda_i=0$ holds for $i\geq n$ thanks to an easy density argument.

Now, for every $1\leq i\leq n$ consider $f_i\in S_{X_i^*}$ and consider 
$$g_i:=(h_n)_{n\in \mathbb N}\in S_{X^*}; h_i=f_i, h_n=0\ \forall n\neq i.$$
Let $\delta>0$ and define $S_i:=S(B_X, g_i,\delta), 1\leq i\leq n$ and $C:=\sum_{i=1}^n \lambda_i S_i$. Given $x=\sum_{i=1}^n \lambda_i x_i\in C$ we have that
$$1-\delta<g_i(x_i)\leq \Vert x_i(i)\Vert.$$
Moreover
$$1\geq \sum_{j=1}^\infty \Vert x_i(j)\Vert ^p=\Vert x_i(i)\Vert ^p+\sum_{j\neq i} \Vert x_j(i)\Vert ^p>(1-\delta)^p+\sum_{j\neq i}\Vert x_i(j)\Vert^p.$$
Thus
$$\sum_{j\neq i} \Vert x_i(j)\Vert^p<1-(1-\delta)^p.$$
Now if we define $y_i\in X$ by $y_i(i)=x_i(i)$ and $y_i(j)=0$ for $j\neq i$ we have
$$\Vert y_i-x_i\Vert^p=\sum_{j\neq i}\Vert x_i(j)\Vert^p<1-(1-\delta)^p.$$
On the other hand, it is immediate that $y_i\in S_i$, so $\sum_{i=1}^n \lambda_i y_i\in C$ and
$$\left\Vert \sum_{i=1}^n \lambda_i (y_i-x_i)\right\Vert< \left(1-(1-\delta)^p\right)^\frac{1}{p}.$$
Furthermore
$$\left\Vert \sum_{i=1}^n \lambda_i y_i\right\Vert^p=\sum_{j=1}^\infty\left\Vert \sum_{i=1}^n \lambda_i y_i(j)\right\Vert^p=\sum_{i=1}^n \lambda_i^p \Vert y_i(i)\Vert^p\leq \sum_{i=1}^n \lambda_i^p<\varepsilon^p.$$
Consequently $\Vert x\Vert\leq \varepsilon+\left(1-(1-\delta)^p\right)^\frac{1}{p}$. The arbitrariness of $x\in C$ proves
$$C\subseteq (\varepsilon+(1-(1-\delta)^p)^\frac{1}{p})B_X.$$
The arbitrariness of $\varepsilon$ and $\delta$ concludes the proof.
\end{proof}
A couple of remarks are pertinent.
\begin{remark}\label{remark:l2sumamejora}

\begin{enumerate}
\item The ideas behind the proof of Proposition~\ref{prop:ccslicearbipe} are inspired in the proof of \cite[Proposition 3.7]{llmr24}.

\item Observe that Proposition~\ref{prop:ccslicearbipe} was implicitly proved in \cite[Theorem 2.12]{ahntt16} under the extra assumption that each $X_n$ contains convex combination of slices of diameter $\leq \varepsilon_n$, where $(\varepsilon_n)\rightarrow 0$. 
\end{enumerate}
\end{remark}

We do not know whether the answer to Question~\ref{question:uniformd2plpsums}. Let us finish the paper pointing out that the idea of using the identification $X_\mathcal U\oplus_p Y_\mathcal U=(X\oplus_p Y)_\mathcal U$ can not be pushed further to infinitely many summands, as the following example shows.

\begin{example}\label{exam:ultrapower}
Let $p\in (1,\infty)\setminus\{2\}$, let $X_n=\mathbb R$ and let $\mathcal U$ be any free ultrafilter over $\mathbb N$. Observe that $(X_n)_\mathcal U=\mathbb R_\mathcal U=\mathbb R$. Thus $\left( \oplus_{n=1}^\infty (X_n)_\mathcal U\right)_p=\ell_p$ isometrically.  However, $\left(\oplus_{n=1}^\infty X_n\right)_p=\ell_p$, and $(\ell_p)_\mathcal U$ is not isometrically isomorphic to $\ell_p$, for instance, since $(\ell_p)_\mathcal U$ contains an isometric copy of $\ell_2$.

Indeed, given any infinite dimensional Banach space $X$, Dvoretzky's theorem implies that $\ell_2$ is finitely representable in $X$ \cite[Theorem 11.3.13]{alka}, so $\ell_2$ is isometrically isomorphic to a subspace of $X_\mathcal U$ thanks to (ii) in \cite[Proposition 11.1.12]{alka}.

\end{example}

\section*{Acknowledgements}

This research has been supported  by MCIU/AEI/FEDER/UE\\  Grant PID2021-122126NB-C31 and by Junta de Andaluc\'{\i}a Grant FQM-0185.

\end{document}